\documentclass[11pt,oneside,reqno]{amsart}
\usepackage{booktabs,amsmath}
\usepackage{amssymb}
\usepackage{amsthm}
\usepackage{graphics}
\usepackage{tikz}
\usepackage{tikz-qtree}
\usepackage{marvosym}
\usetikzlibrary{decorations.pathreplacing}
\usepackage{amsfonts}
\usepackage{mathtools}
\usepackage{algorithm}
\usepackage{mathdots}
\usepackage[noend]{algpseudocode}
\usepackage{float}
\usepackage{caption}
\usepackage{subcaption}
\allowdisplaybreaks
\newcommand{\bt}{\begin{theorem}}
\newcommand{\et}{\end{theorem}}

\usepackage{float}

\usepackage{placeins}

\newtheorem{lemma}{Lemma}
\newcommand{\bl}{\begin{lemma}}
\newcommand{\el}{\end{lemma}}

\newtheorem{corollary}{Corollary}
\newcommand{\bc}{\begin{corollary}}
\newcommand{\ec}{\end{corollary}}

\newtheorem{example}{Example}
\newcommand{\bex}{\begin{example}}
\newcommand{\eex}{\end{example}}

\newtheorem{problem}{Problem} 
\newcommand{\bprob}{\begin{problem}}
\newcommand{\eprob}{\end{problem}}

\usepackage{enumitem}
\newcommand{\beq}{\begin{equation}}
\newcommand{\eeq}{\end{equation}}
\newcommand{\benum}{\begin{enumerate}}
\newcommand{\eenum}{\end{enumerate}}

\newcommand{\N}{\ensuremath{ \mathbf N }}

\newtheorem{rem}{Remark}
\newtheorem{lem}{Lemma}
\newtheorem{cor}{Corollary}
\newtheorem{prop}{Proposition}

\newtheorem{defn}{Definition}
\newtheorem{theorem}{Theorem}

\begin{document}
\title{Limit points of Nathanson's Lambda sequences}
\author{Satyanand Singh}
\date{\today}
\address{Department of Mathematics, New York City College of Technology \newline (CUNY), 300 Jay Street,
Brooklyn, New York 11201}
\email{ssingh@citytech.cuny.edu}


\begin{abstract}
We consider the set $A_{n}=\displaystyle\cup_{j=0}^{\infty}\{\varepsilon_{j}(n)\cdot n^j\colon\varepsilon_{j}(n)\in\{0,\pm1,\pm2,...,\pm\lfloor{{n}/{2}}\rfloor\}\} $. Let  $\mathcal{S}_{\mathcal{A}}= \bigcup_{a \in\mathcal{A} } A_{a}$ where $\mathcal{A}\subseteq \mathbb{N}$. We denote by $\lambda_{\mathcal{A}}(h)$ the smallest positive integer that can be represented as a sum of $h$, and no less than $h$, elements in $\mathcal{S}_{\mathcal{A}}$.  Nathanson studied the properties of the $\lambda_\mathcal{A}(h)$-sequence and posed the problem of finding 
the values of $\lambda_\mathcal{A}(h)$. When $\mathcal{A}=\{2,i\}$, we represent $\lambda_{\mathcal{A}}(h)$  by $\lambda_{2,i}(h)$. Only the values $\lambda_{2,3}(1)=1$, $\lambda_{2,3}(3)=5$,
$\lambda_{2,3}(3)=21$ and $\lambda_{2,3}(4)=150$ are known. In this paper, we extend this result.  For  any odd  $i>1$ and $h\in\{1,2,3\}$, we find the values of $\lambda_{2,i}(h)$. Furthermore,  for fixed $h\in\{1,2,3\}$, we find the values of $\lambda_{2,i}(h)$ that occur infinitely many times as $i$ runs over the odd integers bigger than 1.  We call these numbers the \emph{limit points of Nathanson's lambda sequences}.
 \end{abstract}
\maketitle


\section{Introduction}

We begin with the generating set  
\[ A_{n}=\displaystyle\cup_{j=0}^{\infty}\{\varepsilon_{j}(n)\cdot n^j\colon\varepsilon_{j}(n)\in\{0,\pm1,\pm2,...,\pm\lfloor{{n}/{2}}\rfloor\}\}.\]
For any finite set $\mathcal{A}$ of positive integers bigger than 1 and $\mathcal{S}_{\mathcal{A}}= \bigcup_{a \in\mathcal{A} } A_{a}$, we can represent any positive integer $k$ as a sum of elements of the set  $\mathcal{S}_{\mathcal{A}}$ in the form 
\[
k=\sum_{\alpha\in\mathcal{A}}\left(\sum_{j=0}^{\infty}\varepsilon_{j}(\alpha){\alpha}^{j}\right).
\]
Associated with the above representation of $k$ is the length function which is defined as 
\[
l_{\mathcal{A}}(k)=\sum_{j\in\{0,1,2,...\}, \, {\alpha\in\mathcal{A}}}|\varepsilon_{j}(\alpha)|. 
\]

This representation of integers follows naturally from the special $g$-adic representations of positive integers outlined in sections \ref{a1} and \ref{a2} below. 
We will now proceed to define the $\lambda_\mathcal{A}(h)$ function. 
\begin{defn} For every positive integer $h$, let $\lambda_\mathcal{A}(h)$ denote the smallest positive integer that can be represented as the sum of elements of $\mathcal{S}_{\mathcal{A}}$ with length $h$, but that cannot be represented as a sum with length less than $h$.
\end{defn}
For ease of notation, we will write $\lambda_{\{2,n\}}(h)=\lambda_{2,n}(h)$ and $l_{\{2,n\}}(k)=l_{2,n}(k).$\\

The lambda values are very important in geometric group theory and additive number theory. Nathanson in \cite{nath08pc}, \cite{nath09xx} and \cite{nath12xx} did extensive work in solving some of these problems and derived important properties of the $\lambda_\mathcal{A}(h)$ sequences in the process.  These 
$\lambda_\mathcal{A}(h)$ sequences originate from the study of groups, generators and metric spaces and their existence can be traced back to Jarden and Narkiewicz in \cite{JN15xx} and Hajdu in \cite{HT11xx}. This work was further extended by Hajdu and Tijdeman in \cite{HT12xx} where they found non-trivial bounds for the terms of the $\lambda_{A}(h)$ sequences and in \cite{HT13xx} where they considered sums and differences of power products. Additional extensions were done by Bert\'{o}k in \cite{CS1xx} by considering representations as power products with non-prime bases. Very few terms are known for Nathanson's lambda sequences even in the simplest of cases such as $\lambda_{2,3}(h)$, where one needs to consider representations of integers only as sums and differences of powers of two's and three's. In fact, only the first four terms, $\lambda_{2,3}(1)=1$, $\lambda_{2,3}(2)=5$, $\lambda_{2,3}(3)=21$ and $\lambda_{2,3}(4)=150$ are known; see \cite{nath09xx} and \cite{Singh5xx}. A similar function is considered by Dimitrov and Howe in \cite{DH7xx}, where representations are generated from 
sums and differences of the product of powers of 2's and 3's. In this paper, we will extend work on Nathanson's lambda sequences by finding the values of $\lambda_{2,n}(h)$ for $h\in\{1,2,3\}$ and odd $n>1$. For fixed $h\in\{1,2,3\}$, we also find the values of $\lambda_{2,i}(h)$ that occur infinitely many times as $i$ runs over the odd integers bigger than 1. \\

We will now discuss the case when the cardinality of $\mathcal{A}$, denoted by $|\mathcal{A}|$ is 1, i.e. $\mathcal{A}=\{m\}$ and $\mathcal{S}_{\mathcal{A}}=A_{m}$. This will provide us with a useful tool to compute lambda values. In this case, Nathanson's algorithms in \cite{nath09xx}  allow us to represent positive integers in special $m$-adic form uniquely and with shortest length. We will denote these lengths by $\widehat{l}_{m}(k)$. To find terms of the lambda sequences, we only require shortest length representations, uniqueness is not necessary. This is prevalent for $|\mathcal{A}|>1$ and in particular for representations of integers from $A_{2}\cup A_{n}$ for odd $n>1$. We now outline the algorithms.


\subsection{Generating special $2$-adic representations.}\label{a1}
\benum
\item[(a)] We begin by writing the positive integer $n$ as a standard binary expansion in the unique binary form $m=\sum_{i=0}^{\infty}\varepsilon_i2^i$, where $\varepsilon_i\in\{0,1\}.$ 
\item[(b)] If there are $t$ consecutive sums of powers of 2, i.e. $ \varepsilon_i=\varepsilon_{i+1}=\dots=\varepsilon_{i+t-1}=1$ for $t\ge2,$ then we rewrite $\sum_{j=i}^{i+t-1}2^j=2^{i+t}-2^i$. Notice that this shortens the number of terms from $t$ to 2 for $t>2.$
After performing the above step on all consecutive blocks of powers of 2, if we end up with $s$ consecutive negative terms, of powers of 2 then by a similar argument we see that $ \varepsilon_i=\varepsilon_{i+1}=\dots=\varepsilon_{i+s-1}=-1$ and 
$-\sum_{j=i}^{i+s-1}2^j=-2^{i+s}+2^i$. 
\item[(c)]If we have two consecutive terms with opposite signs, say the $k$ and $k+1$ terms, then we combine them in the following manner: $-2^k+2^{k+1}=2^k$ or $2^k-2^{k+1}=-2^k$  to convert them into a single term. 
\item[(d)]  If identical powers of 2's occur with opposite signs after implementing the above steps, cancel them. If a term of the form $m2^j$ occurs, where $m=2^r$, then coalesce them into $2^{j+r}$. If $m\ne2^j$ then write the binary expansion of $m$ and multiply each power of two in this expansion by $2^j$. 
\eenum

After repeating these steps a finite number of times, there will be no consecutive terms and the process will terminate to give us a minimum special 2-adic expansion of our integer. For example, using the above algorithm we can write $473 =2^9-2^5-2^3+2^0$ and it follows that $\widehat l_{2}(473)=4$.

\subsection{Generating special $g$-adic representations for odd $g>1$.}\label{a2}
\benum
 \item[(a)] Take any positive integer $n$ and write $n$ it in the standard base $g$ representation, $n=\sum_{i=0}^{\infty}\varepsilon_{i}g^{i}$ with 
$\varepsilon_i\in\{0,1,2,...,(g-1)\}$ for all ${i}$.
\item[(b)]If $(g+1)/2\leq\varepsilon_{i}\leq g-1$ for some $i$, then choose the smallest such $i$ and apply the identity 
\[
\varepsilon_{i}g^{i}= - (g-\varepsilon_{i})g^{i}+g^{i+1}.
\]
\item[(c)] If $-(g-1)\leq\varepsilon_{i}\leq-(g+1)/2$ for some $i$, then apply the identity 
\[
\varepsilon_{i}g^{i}= -(g+\varepsilon_{i})g^{i}-g^{i+1}.
\]
\item[(d)] Delete all occurrences of $0$.
\item[(e)] If $g^{i}$ and $-g^{i}$ occur, delete them.
\item[(f)] If $g^{i}$ (resp $-g^{i}$) occurs $g$ times, replace these $g$ summands with the one summand $g^{i+1}$ (resp $-g^{i+1}$).\\
\eenum

By iterating the above process, we get a special $g$-adic representation of $n$. 
For example 39797 written as a special 9-adic representation is $39797=1\cdot9^{5}-3\cdot9^{4}+1\cdot9^{3}-4\cdot9^{2}+3\cdot9^{1}-1\cdot9^{0}$ and it follows that $\widehat l_{9}(39797)=13$.

In the next section we will examine representations of positive integers with elements from $\mathcal{S}_{\{ 2,n\}}$ for odd $n>1$. In particular, we will find the values of 
$\lambda_{2,n}(h)$ for $h=1,2$ and $3$. 

\section{Computing $\lambda_{2,n}(h)$ for $h\in\{1,2,3\}$ and odd $n>1$.}
It is clear that representations of numbers from $A_{2}$ and $A_{n}$ are also representations from $A_{2,n}$. In addition $\widehat{l}_{2,n}(k) \leq \widehat l_{2}(n)$, where $\widehat{l}_{2,n}(k)$ denote the minimal length of $k$ in $\mathcal{S}_{\{ 2,n\}}=A_{2}\cup A_{n}$. 
We will now proceed to compute the values of $\lambda_{2,n}(h)$ for $h\in\{1,2,3\}$ and any odd $n>1$.
We begin with an example to compute some lambda values. 
\begin{example}
$\lambda_{2,5}(1)=1$, $\lambda_{2,5}(2)=3$ and $\lambda_{2,5}(3)=19$. 
\end{example}

We represent  the first twenty positive integers using elements from
\[
A_{2}\cup A_{5}=\{0\} \cup \{\, \pm{2^j, j=0,1,2,3,...}\}\cup\{\pm1\cdot 5^m:m=0,1,2,...\}\cup\{\pm2\cdot 5^k:k=0,1,2,...\}
\]
and list their corresponding lengths in Table \ref{table:Sat25}. Our task is to identify the first number with a representation of shortest length one to get $\lambda_{2,5}(1)$, which is trivially 1. Similarly we identify the first number with shortest length 2, to get $\lambda_{2,5}(2)$ and so on.

We observe that we can represent 1 and 2 as numbers with lengths 1; namely $1=2^0$ and $2=2^1$. Clearly $3$ cannot have length 1, but it has several representations with length 2. This readily follows since the equations ${2^{\gamma_{1}}}=3, \,\, {2\cdot5^{\gamma_{2}}}=3\,\, \text{and} \,\, {5^{\gamma_{3}}}=3$ are insoluble in $\bf{N}$, the set of positive integers and one representation of $3$ with length 2 is $-1\cdot2^1+1\cdot5^1$. The first number with length three is 19. To see this we need to establish that $l_{2,5}(19)\geq3$, and as such that ${\pm2^{\alpha_{1}}}=19$, ${\pm2^{\alpha_{2}}\pm2^{\alpha_{3}}}=19$, ${\pm5^{\alpha_{4}}}=19$, ${\pm2\cdot5^{\alpha_{5}}}=19$, ${\pm5^{\alpha_{6}}\pm5^{\alpha_{7}}}=19$, ${\pm2^{\alpha_{8}}\pm5^{\alpha_{9}}}=19$ are all insoluble. By Nathanson's algorithms, we see that $\widehat l_{2}(19)=\widehat l_{2}(1\cdot2^4+1\cdot2^1-1\cdot2^0)=3$ and $\widehat l_{5}(19)=\widehat l_{2}(1\cdot5^2-1\cdot5^1-1\cdot5^0)=3$, which establish the insolubility of all the equations except for ${\pm2^{\alpha_{8}}\pm5^{\alpha_{9}}}=19$ . This latter system's insolubility is established in the proof of Theorem \ref{t3.2} and it follows that $\widehat l_{2,5}(19)=3$.

\FloatBarrier
\begin{table}[!htp]
\caption{Representations in $A_{2}\cup A_{5}$.}
\begin{center}
\begin{tabular}{|l |c||l|c|}
\hline
$n$&$l_{2,5}(n)$&$n$&$l_{2,5}(n)$\\
\hline
1=$1\cdot 2^{0}$&\boxed{1}&$11=1\cdot2^4-1\cdot5^1 $& 2 \\
2=$1\cdot2^1$&1&$12=1\cdot2^4-1\cdot2^2 $& 2 \\
3=$-1\cdot 2^0+1\cdot2^2 $& \boxed{2}&$13=1\cdot2^3+1\cdot5^1 $& 2  \\
4=$1\cdot2^2 $& 1&$14=1\cdot2^4-1\cdot2^1 $& 2\\ 
5=$1\cdot5^1 $& 1&$15=1\cdot 2^{4}-1\cdot2^0$&2\\
6=$-1\cdot2^1+1\cdot2^3 $& 2&$16=1\cdot2^4$&1 \\
7=$1\cdot2^1+1\cdot5^1 $& 2&$17=1\cdot2^4+1\cdot 2^0$& 2\\
8=$1\cdot 2^{3}$&1&$18=1\cdot2^4+1\cdot2^1 $& 2 \\
9=$1\cdot2^0+1\cdot2^3$&2&$19=-1\cdot2^0-1\cdot5^1+1\cdot5^2 $& \boxed{3} \\
10=$2\cdot 5^1$& 2&$20=1\cdot2^4+1\cdot 2^2$& 2  \\
\hline
\end{tabular}
\end{center}
\label{table:Sat25}
\end{table}
\[
\]
We will now prove some results that are essential to the establishment of $\lambda_{2,n}(h)$ for $h\in\{1,2,3\}$ and odd $n>1$. \\

\begin{rem}\label{l3.1}{\it Let P be a finite or infinite set of integers $>1$. For all $k\in P$ it readily follows from $k=1\cdot k^{1}$ that $\widehat l_{P}(k)={1}$}.
  
 \end{rem} 

For the rest of the paper, we will write representations of the form $\pm1\cdot a^n$ as $\pm a^n$.

\begin{lem}\label{l3.2} {\it $\widehat l_{2,n}(t)\leq2$ for $t\in Q$ where,} 
\[
Q=\{1,2,3,4,5,6,7,8,9,10,12,14,15,16,17,18,20,24\}.
\]
\end{lem}
\begin{proof}
For $t\in Q$, $\widehat l_{2}(t)\leq{2}$ which is seen from the special $2$-adic algorithm in Section \ref{a1}. 
We list the shortest length representations of the elements of $Q$ from $A_{2}$. $1=1, 2=2, 3=-1+2^2, 4=2^2, 5=1+2^2, 6=-2+2^3, 7=-1+2^3, 8=2^3, 9=1+2^3,10=2^1+2^3, 12=-2^2+2^4, 14=-2+2^4, 15=-1+2^4, 16=2^4, 17=1+2^4, 18=2+2^4, 20=2^2+2^4, 24=-2^3+2^5$. It readily follows that we can find representations of $t$ such that $\widehat l_{2,n}(t)\leq2$ since $\widehat l_{2,n}(t)\leq \widehat l_{2}(t)$
\end{proof}

We will pay particular attention to the number $11$ and consider its representations and length in the special $n$-adic form for odd $n$. We will classify all odd $n'$s with $\widehat l_{2,n}(11)\leq{2}$.\\

\begin{lem}\label{l3.4}{\it  For odd $n>1$, $\widehat l_{2,n}(11)=1$ if and only if $n=11$. Further, $\widehat l_{2,n}(11)=2$ for odd $n'$s if and only if $n\in S$, where $S_{1}=\{w:w=-2^{\gamma}+11,\gamma=1,2,3\}$, $S_{2}=\{t:t=2^{\alpha}+11,\alpha\geq{1}, \alpha\in\N\}$, $S_{3}=\{s:s={2^{\beta}-11}, \beta\geq{4},\beta\in\N\}$ and $S=S_{1}\cup S_{2}\cup S_{3}=\{3, 5, 7, 9, 13, 15, 19, 21, 27, 43, 53, 75, ...\}$.}
\end{lem}
\begin{proof}
The first statement is trivial. 
For the second part, when $u\in S$ it readily follows that $\widehat l_{2,u}(11)=2$, since the elements in $S$ satisfy the equation 11=$|2^c\pm u|$ and from the first part we see that $\widehat l_{2,u}(11)\neq1$.
\end{proof}

\begin{lem}\label{l3.5} 
{\it For odd $n>1$, $\widehat l_{2,n}(13)=1$ if and only if $n=13$. Further, $\widehat l_{2,n}(13)=2$ for odd $n'$s if and only if $n\in T$, where $T_{1}=\{l:l=-2^{m_{1}}+13, m_{1}=1, 2, 3\}$, $T_{2}=\{m:m=2^{\gamma_1}+13,\gamma_ 1\geq{1}, \gamma_1\in\N\}$, $T_{3}=\{k:k=2^{\phi +2}-13, \phi\geq{4},\phi\in\N\}$ and $T=T_{1}\cup T_{2}\cup T_{3}=\{3,5,9,11,15,17,19,21,29,45,51,77,...\}$.} 

\end{lem}

\begin{proof}
The first statement is trivial. For the second part, when $v\in T$ it readily follows that $\widehat l_{2,v}(13)=2$, since the elements in $T$ satisfy the equation 13=$|2^d\pm v|$ and from the first part we see that $\widehat l_{2,u}(13)\neq1$.
\end{proof}

\begin{lem}\label{l3.6}
{\it $T\cap S=\{3, 5, 9, 15, 19, 21\}$.} 
\end{lem}
\begin{proof}
This intersection follows by finding all the solutions of the equations below.

\begin{equation}
n_{1}=2^{c_1} +13=2^{c_2}+11\,\ \text{or}\,\  2^{c_2}-2^{c_1}=2
\end{equation}
\[
\]
\begin{equation}
n_{2}=2^{d_1} +13=2^{d_2}-11\,\ \text{or}\,\  2^{d_2}-2^{d_1}=24
\end{equation}
\[
\]
\begin{equation}
n_{3}=2^{g_1} +13=-2^{g_2}+11\,\ \text{or}\,\  2^{g_1}+2^{g_2}=-2
\end{equation}
\[
\]
\begin{equation}
n_{4}=2^{e_1} -13=2^{e_2}+11\,\ \text{or}\,\  2^{e_1}-2^{e_2}=24
\end{equation}
\[
\]
\begin{equation}
n_{5}=2^{f_1} -13=2^{f_2}-11\,\ \text{or}\,\  2^{f_1}-2^{f_2}=2
\end{equation}
\[
\]
\begin{equation}
n_{6}=2^{h_1} -13=-2^{h_2}+11\,\ \text{or}\,\  2^{h_1}+2^{h_2}=24
\end{equation}
\[
\]
\begin{equation}
n_{7}=-2^{u_1} +13=2^{u_2}+11\,\ \text{or}\,\  2^{u_1}+2^{u_2}=2
\end{equation}
\[
\]
\begin{equation}
n_{8}=-2^{v_1} +13=2^{v_2}-11\,\ \text{or}\,\  2^{v_1}+2^{v_2}=24
\end{equation}
\[
\]
\begin{equation}
n_{9}=-2^{w_1} +13=-2^{w_2}+11\,\ \text{or}\,\  2^{w_1}-2^{w_2}=2
\end{equation} 
\[
\]
It follows from the above equations that 
\[
\]
\begin{itemize}
\item $(n_{1},c_{1},c_{2})=(15,1,2)$, $(n_{2},d_{1},d_{2}) =(21,3,5)$, $(n_{4},e_{1},e_{2})=(19,5,3)$ \\

\item $(n_{5},f_{1},f_{2})=(-9,2,1)$, $(n_{6},h_{1},h_{2})=(3,4,3)$, $(n_{6},h_{1},h_{2})=(-5,3,4)$\\

\item $(n_{7},u_{1},u_{2})=(12,0,0)$, $(n_{8},v_{1},v_{2})=(5,3,4)$, $(n_{9},w_{1},w_{2})=(9,2,1)$\\

\end{itemize}

The positive odd solutions are $\{3, 5, 9, 15, 19, 21\}$. We now establish that there are no other solutions.\\

For equation (7), $2^{u_1}+2^{u_2}=2$, we can consider $2^{u_1}=2-2^{u_2}$, which has the solution $u_1=u_2=0$ and it is clearly insoluble for $u_{2}>1$, since we get $2^{u_1}<0$.\\

For the equations of the shape $2^m-2^n=2$, namely equations (1), (5), (9), clearly $m>n$ and we can write $m=n+\delta$ from which it follows that $2^{n}(2^{\delta}-1)=2$ and the only solution occurs when $n=1$ and $\delta=1$.\\

For the equations of the shape $2^m+2^n=24$, namely equations (6), (8), we can consider $2^m=24-2^n$ which has the solutions $m=3, n=4$ and $m=4, n=3$. It is clearly insoluble for $n>4$, since we get $2^m<0$.\\

For the equations of the shape $2^m-2^n=24$, namely equations (2), (19), clearly $m>n$ and we can write $m=n+\delta$ from which it follows that $2^{n}(2^{\delta}-1)=2^3\cdot3$ and the only solution occurs when $n=3$ and $\delta=2$.\\

For equation (3), $2^{g_{1}}+2^{g_{2}}=-2$, is clearly insoluble.
\end{proof}

\begin{prop}\label{p3.1} 
{\it $\widehat l_{2,n}(11)=3$ for any $n\notin S$.} 
\end{prop}
\begin{proof}
For any $n\notin S$, by Lemmas \ref{l3.2} and \ref{l3.4} and the fact that $\widehat l_{2,n}(11)=\widehat l_{2,n}(-1-2^2+2^4)=3$ the claim follows.

\end{proof}

\begin{prop}\label{p3.2} 
{\it $\widehat l_{2,n}(13)=3$ for $n\in S\smallsetminus\{3, 5, 9, 15, 19, 21\}$.} 
\end{prop}
\begin{proof}
For any $n\notin S$, by Lemmas \ref{l3.2}, \ref{l3.4}, \ref{l3.5} and \ref{l3.6} and the fact that $\widehat l_{2}(13)=\widehat l_{2}(1-2^2+2^4)=3$, the claim follows.
\end{proof}

We will now proceed to establish $\lambda_{2,n}(h)$ for $h\in\{1,2,3\}$ and odd $n>1$. It is a routine calculation to show that $\lambda_{2,3}(2)=5$, see for example \cite{nath09xx}.\\

\begin{prop}\label{p3.3}
{\it $\lambda_{2,n}(1)=1$, and $\lambda_{2,n}(2)=3$ for any odd integer $n\geq5$, $n\in\N$.} 
\end{prop}
\begin{proof}
$\lambda_{2,n}(1)=1$ is trivial and $\lambda_{2,n}(2)=3$ follows from the fact that $\widehat l_{2}(1)=\widehat l_{2}(2)=1$, $\widehat l_{2}(3)=2$ and the Diophantine equations: $\pm2^a=3$, and $\pm{n^b}=3$ are insoluble. 
\end{proof}

\begin{cor}\label{c1}
{The limit point of \it $\lambda_{2,n}(1)$ is $1$ and that of $\lambda_{2,n}(2)$ is $3$.} 
\end{cor}
\begin{proof}
The corollary follows immediately from Proposition \ref{p3.3}.
\end{proof}

We have seen in Lemma \ref{l3.6} that $T\cap S=\{3, 5, 9, 15, 19, 21\}$ and that $\{11,13\}$ are also points of interest since $\widehat l_{2,11}(11)=1$, $\widehat l_{2,11}(13)=\widehat l_{2,11}(2+11)=3$, $\widehat l_{2,13}(13)=1$ and $\widehat l_{2,13}(11)=\widehat l_{2,13}(-2+13)=2$ which readily follows from Remark \ref{l3.1} and the insolubility of $\pm13^t=11$ and $\pm 11^s=13$. It follows that in the computation of $\lambda_{2,n}(3)$ we need to consider the set of outliers for  $n\in\{3, 5, 9,11, 13, 15, 19, 21\}$ separately.\\

We will now state our main result which will allow us to find all limit points of $\lambda_{2,n}(3)$.

\begin{theorem}\label{t3.1} 
{{\it {For odd $n>1$, where S is the set defined in Lemma \ref{l3.4} we have that}}  \\
\[
\lambda_{2,n}(3)=
	\left \{
		\begin{array}{ll}
				11\,\,  \text{for }\,\, n\notin S\cup \{11\}, \\
				13\,\,  \text{for }\,\, n\in{S\smallsetminus\{3, 5, 9,11, 13, 15, 19, 21\}}.\\
						\end{array}
		\right.		
\]}
\end{theorem}
\begin{proof}
From Lemma \ref{l3.2}, we have that $\widehat l_{2,n}(t)\leq2$ for $t\in\{1,2,3,4,5,6,7,8,9,10,12\}$. Using this fact and  Propositions \ref{p3.1} and \ref{p3.2} our Theorem follows.
\end{proof}

\begin{cor}\label{c2}
{The limit points of \it $\lambda_{2,n}(3)$ are $11$ and $13$.} 
\end{cor}
\begin{proof}
The corollary follows immediately from Theorem \ref{t3.1} and $|S|=\infty$.
\end{proof}

\begin{theorem}\label{t3.2}
{{\it {$\lambda_{2,3}(3)=21$, $\lambda_{2,5}(3)=19$,   
$\lambda_{2,9}(3)=19$,  $\lambda_{2,11}(3)=23$,  $\lambda_{2,13}(3)=22$,  $\lambda_{2,15}(3)=21$, $\lambda_{2,19}(3)=22$, $\lambda_{2,21}(3)=26$.}}}
\end{theorem}
\begin{proof}
$\lambda_{2,3}(3)=21$, see \cite{nath09xx}. Now consider the pairs $(g,k_{2,g})$, where
\[
(g, k_{2,g})\in\{(5,19), (9,19), (11, 23), (13,22), (15,21), (19, 22), (21,26)\}.
\]

For integer valued $x$, where $1\leq{x}<k_{2,g}$ we will first establish that  $l_{2,g}(x)\leq2$. An application of Lemmas \ref{l3.2} and \ref{l3.6}, shows that $\widehat l_{2,g}(x)\leq2$ for $x\in \widehat Q=Q \cup\{11, 13\}$. For each $g \in\{5, 9, 15, 19, 21\}$ we will now consider the values of $x\notin \widehat Q$ but in the interval $1\leq{x}<k_{2,g}$ and their appropriate representations to show that $\widehat l_{2,g}(x)\leq2$. 
\[
\]
\begin{itemize}
\item For $g \in \{5,9\}$, $x\in  \widehat Q$ includes all values in the interval $1\leq{x}<19$.\\
\item For $g=11$, $x \in \{11,19=2^3+11, 21=2^5-11, 22=2\cdot11\}\cup\widehat Q$ shows that $\widehat l_{2,11}(x)\leq2$ for values in the interval $1\leq{x}<23$.\\
\item For $g=13$, $x \in \{13,19=2^5-13, 21=2^3+13\}\cup\widehat Q$ shows that $\widehat l_{2,13}(x)\leq2$ for values in the interval $1\leq{x}<22$.\\
\item For $g=15$, $\{x=19=2^2+15\}\cup\widehat Q$ shows that $\widehat l_{2,15}(x)\leq2$ for values in the interval $1\leq{x}<21$.\\
\item For $g=19$, $x \in \{19,21=2+19\}\cup\widehat Q$ shows that $\widehat l_{2,19}(x)\leq2$ for values in the interval $1\leq{x}<22$.\\
\item For $g=21$, $x \in \{19=-2+21, 21, 22=1+21, 23=2+21, 25=2^2+21\}\cup\widehat Q$ shows that $\widehat l_{2,21}(x)\leq2$ for values in the interval $1\leq{x}<26$.\\
\end{itemize}
\[
\]
We will now establish that the Diophantine equations listed below which are representations of  $k_{2,g}$ with $l_{2,g}(k_{2,g})=2$ are insoluble in non-negative integers for each pair $(g,k_{2,g})$ to rule out the possibility that $l_{2,g}(k_{2,g})\leq 2$. In addition, we will produce a representation of $k_{2,g}$ with $l_{2,g}(k_{2,g})=3$ to show that $\widehat l_{2,g}(k_{2,g})=3$ and complete the proof of our Theorem.

\begin{equation}
\pm2^{a_1}=k_{2,g}
\end{equation}
\[
\]
\begin{equation}
\pm g^{a_2}=k_{2,g}
\end{equation}
\[
\]
\begin{equation}
\pm2\cdot g^{a_3}=k_{2,g}
\end{equation}
\[
\]
\begin{equation}
\pm2^{a_4} \pm2^{a_{5}}=k_{2,g}
\end{equation}
\[
\]
\begin{equation}
\pm g^{a_6} \pm g^{a_{7}}=k_{2,g}
\end{equation}
\[
\]
\begin{equation}
\pm2^{a_8}\pm g^{a_{9}}=k_{2,g}  
\end{equation}
\[
\]
We now compute $\widehat l_{2}(k_{2,g})$ and $\widehat l_{g}(k_{2,g})$ and display in Table \ref{table:sat14} below.
\[
\]
\FloatBarrier
\begin{table}[!htp]
\caption{$k_{2,g}$ and its corresponding $\widehat l_{2}(k_{2,g})$ and $\widehat l_{g}(k_{2,g})$.} 
\begin{center}
\begin{tabular}{|c|c|r|r|c|c|}
\hline 
$g$&$k_{2,g}$&Representations in $A_{2}$&Representations in $A_{g}$ &$\widehat l_{2}(k_{2,g})$&$\widehat l_{g}(k_{2,g})$  \\ \hline 
$5$&$19$&$-1+2^2+2^4$&$-1-5+5^2$&$3$&$3$ \\ \hline
$9$&$19$&$-1+2^2+2^4$&$1+2\cdot9$&$3$&$3$ \\ \hline
$11$&$23$&$-1-2^3+2^5$&$1+2\cdot11$&$3$&$3$ \\ \hline
$13$&$22$&$2^5-2^3-2$&$-2^2+2\cdot13$& $3$&$3$  \\ \hline
$15$&$21$&$1+2^2+2^4$&$6+15$&$3$&$7$  \\ \hline
$19$&$22$&$-2-2^3+2^5$&$3+19$&$3$&$4$  \\ \hline
$21$&$26$&$2-2^3+2^5$&$5+21$&$3$&$6$  \\ \hline
 \end{tabular}
\end{center}
\label{table:sat14}
\end{table}
 
Observe that $\widehat l_{2}(k_{2,g})=3$ and $\widehat l_{g}(k_{2,g})\geq3$, and as a consequence of Nathanson's minimal representations discussed in sections 1.1 and 1.2, equations represented by $(10)$ through $(14)$ above are insoluble.
We will now consider the seven remaining cases (twenty one equations) for $g\in\{5,9,11,13,15,19,21\}$ individually to establish insolubility of equations represented by (15) above by finding obstructions in certain rings.\\

\begin{enumerate}
\item[I.](a) For $2^{c_1}+5^{c_2}=19$, we get insolubility by substituting for $c_1=0,1,2,3,4$.\\
(b) For $2^{c_3}-5^{c_4}=19$, we get obstructions in the ring $\mathbb{Z}/{15}\mathbb{Z}$.\\
(c) For $-2^{c_5}+5^{c_6}=19$. Clearly for $c_{5}=0,1$ the equation is insoluble. For $c_{5}>1$ we get obstructions in the ring $\mathbb{Z}/{4}\mathbb{Z}$.\\
\item[II.](a) For $2^{e_1}+9^{e_2}=19$, we get insolubility by substituting for $e_1=0,1,2,3,4$.\\
(b) For $2^{e_3}-9^{e_4}=19$, when $e_{3}=0,1$ or for $e_{4}=0$ there are no solutions and for $e_{3}>2$ and $e_{4}>0$, we get obstructions in the ring $\mathbb{Z}/{8}\mathbb{Z}$.\\
(c) For $-2^{e_5}+9^{e_6}=19$, clearly for $e_{6}=0,1$ the equation is insoluble and for $d_{6}>0$ we get obstructions in the ring $\mathbb{Z}/{4}\mathbb{Z}$.\\

\item[III.](a) For $2^{f_1}+11^{f_2}=23$, we get insolubility by substituting for $f_1=0,1,2,3,4$. We  see that there are no corresponding values for $f_2$. \\
(b) For $-2^{f_3}+11^{f_4}=23$, we get obstructions in  $\mathbb{Z}/{91}\mathbb{Z}$.\\
(c) For $2^{f_3}-11^{f_4}=23$, we get obstructions in  $\mathbb{Z}/{133}\mathbb{Z}$.\\

\item[IV.](a) For $2^{g_1}+13^{g_2}=22$, we get insolubility by substituting for $g_1=0,1,2,3,4$.\\
(b) For $2^{g_3}-13^{g_4}=22$, we see that there is no solution for ${g_3}=0$ or ${g_4}=0$ and we get obstructions in the ring $\mathbb{Z}/{2}\mathbb{Z}$ for ${g_3}>0$ or ${g_4}>0$.\\
(c) For $-2^{g_5}+13^{g_6}=22$, clearly for $g_{5}=0$ or for $g_{6}=0$, the equation is insoluble and for $g_{5}>0$  and $g_{6}>0$ we get obstructions in the ring $\mathbb{Z}/{2}\mathbb{Z}$.\\

\item[V.](a) For $2^{m_1}+15^{m_2}=21$, we get insolubility by substituting for $m_1=0,1,2,3,4$.\\
(b) For $2^{m_3}-15^{m_4}=21$, we see that there is no solution for ${k_4}=0$. We get obstructions in the ring $\mathbb{Z}/{15}\mathbb{Z}$ for ${k_4}>0$.\\
(c) For $-2^{k_5}+15^{k_6}=21$, clearly for $k_{6}=0$, the equation is insoluble. For $k_{6}>0$  we get obstructions in the ring $\mathbb{Z}/{15}\mathbb{Z}$.\\

\item[VI.](a) For $2^{u_1}+19^{u_2}=22$, we get insolubility by substituting for $u_1=0,1,2,3,4$.\\
(b) For $2^{u_3}-19^{u_4}=22$, we see that there is no solution for ${u_3}=0$ or ${u_4}=0$. We get obstructions in the ring $\mathbb{Z}/{2}\mathbb{Z}$ for ${u_3}>0$ or ${u_4}>0$.\\
(c) For $-2^{u_5}+19^{u_6}=22$, clearly for $u_{5}=0$ or for $u_{6}=0$, the equation is insoluble. For $u_{5}>0$  and $u_{6}>0$ we get obstructions in the ring $\mathbb{Z}/{2}\mathbb{Z}$.\\

\item[VII.](a) For $2^{z_1}+21^{z_2}=26$, we get insolubility by substituting for $z_1=0,1,2,3,4$.\\
(b) For $2^{z_3}-21^{z_4}=26$, we see that there is no solution for ${z_3}=0$ or ${z_4}=0$. We get obstructions in the ring $\mathbb{Z}/{2}\mathbb{Z}$ for ${z_3}>0$ or ${z_4}>0$.\\
(c) For $-2^{z_5}+21^{z_6}=26$, clearly for $z_{5}=0$ or for $z_{6}=0$, the equation is insoluble. For $z_{5}>0$  and $z_{6}>0$ we get obstructions in the ring $\mathbb{Z}/{2}\mathbb{Z}$.\\
\end{enumerate}
 
We have shown that $\widehat l_{2,g}(k_{2,g})=3$ and it follows that $\lambda_{2,3}(3)=21$, $\lambda_{2,5}(3)=19$, $\lambda_{2,9}(3)=19$, $\lambda_{2,11}(3)=23$,  $\lambda_{2,13}(3)=22$,  $\lambda_{2,15}(3)=21$, $\lambda_{2,19}(3)=22$, $\lambda_{2,21}(3)=26$.
\end{proof}

\section{Additional results and open problems.}

In \cite{Singh5xx} it is shown that for any fixed odd $g\geq3$, $h\in\N$ and $j=1,2,3,...,(g-1)/2$, 
\[
\lambda_{g}\left(\frac{(h-1)(g-1)+2j}{2}\right)=\frac{(2j-1)g^{h-1}+1}{2}.
\] 
An explicit expression is also established for $\lambda_{2}(h)$ and explicit upper bounds are given for $\lambda_{g}(h)$, where $g$ is an even integer greater than $2$. It is also shown that $\lambda_{2,3}(4)=150$ and $\lambda_{2,5}(4)=83.$\
It is an open problem to find additional terms of $\lambda_{2,n}(h)$ for odd $n>1$ and $h\geq4$.\\

\subsection*{Acknowledgement}
I would like to thank my advisor Melvyn B. Nathanson for introducing me to this problem.

\end{document}